\documentclass[10pt]{article}

\usepackage{amsmath,amsthm,fullpage,amssymb,graphicx,bm,enumerate,color}
\usepackage{fourier,eucal}
\usepackage[bbgreekl]{mathbbol}
\usepackage{tikz,subfig}

\newcommand{\tst}{\textstyle}

\newcommand{\m}[1]{\mathcal{#1}}
\newcommand{\bb}[1]{\mathbb{#1}}
\newcommand{\ur}[1]{\mathrm{#1}}

\newcommand{\eps}{\varepsilon}

\newcommand{\oh}{\frac{1}{2}}

\newcommand{\R}{\bb R}
\newcommand{\N}{\bb N}
\newcommand{\wc}{\rightharpoonup}

\renewcommand{\d}{\ur{d}}
\newcommand{\p}{\partial}

\renewcommand{\theta}{{\vartheta}}
\renewcommand{\O}{{\Omega}}
\newcommand{\oO}{\overline{\O}}
\newcommand{\pO}{\p \O}

\newcommand{\si}[1]{{\mathsf{#1}}_i}
\newcommand{\sia}[1]{{\mathsf{#1}}_i^\alpha}

\newcommand{\siw}[1]{{\mathsf{#1}}_i^w}
\newcommand{\siwk}[1]{{\mathsf{#1}}_i^{k,w}}
\newcommand{\hil}[1]{\hat{#1}_i^\ell}
\newcommand{\ika}[1]{{#1}_i^\kappa}
\newcommand{\ila}[1]{{#1}_i^\lambda}
\newcommand{\il}[1]{{#1}_i^\ell}
\newcommand{\ik}[1]{{#1}_i^k}

\newcommand{\iko}[1]{{#1}_i^{k+1}}
\newcommand{\ia}[1]{{#1}_i^\alpha}
\newcommand{\bia}[1]{\bar{#1}_i^\alpha}
\newcommand{\bbia}[1]{\bar{\bar{#1}}_i^\alpha}
\newcommand{\bial}[1]{\bar{#1}_i^{\alpha,\ell}}
\newcommand{\bbial}[1]{\bar{\bar{#1}}_i^{\alpha,\ell}}
\newcommand{\Id}{{\sf Id}}

\renewcommand{\P}{{P_i}}

\newcommand{\bllangle}{{\bigl\langle \hspace{-0.9mm} \bigr\langle}}
\newcommand{\brrangle}{{\bigl\rangle \hspace{-0.9mm} \bigr\rangle}}
\newcommand{\llangle}{{\langle \!\! \langle}}
\newcommand{\rrangle}{{\rangle \!\! \rangle}}

\newtheorem{definition}{Definition}
\newtheorem{theorem}{Theorem}

\newtheorem{corollary}{Corollary}
\newtheorem{lemma}{Lemma}

\newtheorem{assumption}{Assumption}

\newtheorem{algorithm}{Algorithm}
\newtheorem{example}{Example}

\renewcommand{\a}{\alpha}

\newcommand{\abs}[1]{\left|#1\right|}

\newcommand{\tends}{\rightarrow}

\newcommand{\norm}[1]{\left\|#1\right\|}

\newcommand{\half}{\frac{1}{2}}
\newcommand{\ie}{i.e.\ }

\newcommand{\supp}{\mathrm{supp}\,}

\newcommand{\eval}[2]{ #1\vert_{#2}}
\newcommand{\pair}[2]{\langle #1,#2 \rangle}

\newcommand{\RN}{\mathbb{R}^n }

\newcommand{\calT}{\mathcal{T}}

\newcommand{\calI}{\mathcal{I}}

\newcommand{\tia}[1]{\tilde{#1}_i^{\a}}
\newcommand{\ttia}[1]{\tilde{\tilde{#1}}_i^{\a}}

\newcommand{\Eb}{\begin{equation}}
\newcommand{\Ee}{\end{equation}}
\newcommand{\eb}{\begin{equation*}}
\newcommand{\ee}{\end{equation*}}

\begin{document}

\begin{titlepage}
\title{\bf On the Convergence of Finite Element Methods for Hamilton-Jacobi-Bellman Equations}
\author{Max Jensen\footnote{Department of Mathematical Sciences, University of Durham, South Road, Durham DH1 3LE, England, m.p.j.jensen@durham.ac.uk} ,
        Iain Smears\footnote{Mathematical Institute, University of Oxford, 24-29 St.~Giles', Oxford OX1 3LB, England, iain.smears@maths.ox.ac.uk}}
\end{titlepage}

\maketitle

\begin{abstract}
In this note we study the convergence of monotone $P1$ finite element methods on unstructured meshes for fully non-linear Hamilton-Jacobi-Bellman equations arising from stochastic optimal control problems with possibly degenerate, isotropic diffusions. Using elliptic projection operators we treat discretisations which violate the consistency conditions of the framework by Barles and Souganidis. We obtain strong uniform convergence of the numerical solutions and, under non-degeneracy assumptions, strong $L^2$ convergence of the gradients.
\end{abstract}

\section{Introduction}

Hamilton-Jacobi-Bellman (HJB) equations, which are of the form
\begin{align} \label{eq:HJBintro}
-\partial_t v + \sup_\alpha (L^\alpha v - d^\alpha) = 0,
\end{align}
where the $L^\alpha$ are linear first- or second order operators and $d^\alpha \in L^2$, characterise the value function of optimal control problems. Indeed, one possibility to introduce the notion of solution of \eqref{eq:HJBintro} is via the underlying optimal control structure. An alternative approach is to use the monotonicity properties of the operator which leads to the concept of viscosity solutions. While these perceptions are essentially equivalent \cite[p.72]{FS06} both views have been instructive for the design and analysis of numerical methods.

The former approach, based on the discretisation of the optimal control problem before employing the Dynamic Programming Principle, has been proposed in the setting of finite elements in \cite{Q81,CF95,CJ09}, see also the review article \cite{K90} and the references therein. Regarding finite difference methods we refer to the book \cite{KD01}. The latter approach, which is also adopted in this note, was firmly established with the contribution \cite{BS91} by Barles and Souganidis in 1991, providing an abstract framework for the convergence to viscosity solutions. Starting with \cite{K97,K00} techniques were developed to quantify the rate of convergence; more recent works are \cite{BJ07,DK07}. A third direction was opened by the method of vanishing moments which neither enforces discrete maximum principles nor makes use of the underlying optimal control structure but relies on a higher order regularisation \cite{FN11}. For a more comprehensive review of the state-of-the-art in the numerical solution of fully non-linear second order equations we refer to \cite{FGN11}.

In the traditional finite element analysis the multiplicative testing with hat functions is viewed as the discrete analogue of the multiplicative testing procedure to define weak solutions of the (variational) differential equation. While elements of this viewpoint are implicitly used in Section VII on gradient convergence, we would like to stress a second interpretation: multiplication with hat functions as regularisation of the residual. Consider for a moment the linear problem $- a(x) \Delta u(x) = f(x)$ with smooth functions $a$ and $u$ as well as a hat function $\phi$ at the node $y^\ell$. Let $P$ be the orthogonal projection onto the approximation space with respect to the scalar product $\langle v, w \rangle = \int \nabla v \cdot \nabla w \, \d x$ (given suitable boundary conditions). If $y$ is near $y^\ell$ then on a fine mesh
\[
- a(y) \Delta u(y) = - \int a(y) \Delta u(y) \, \hat{\phi}(x) \, \d x \approx - a(y^\ell) \int \Delta u(x) \, \hat{\phi}(x) \, \d x =  a(y^\ell) \int \nabla u(x) \cdot \nabla \hat{\phi}(x) \, \d x = a(y^\ell) \int \nabla P u(x) \cdot \nabla \hat{\phi}(x) \, \d x.
\]
since $\hat{\phi} := \phi / \| \phi \|_{L^1(\O)}$ approximates a Dirac Delta as the element size is decreased. In contrast, on general meshes,
\[
- a(y) \Delta u(y) \not\approx a(y^\ell) \int \nabla \m I u(x) \cdot \nabla \hat{\phi}(x) \, \d x, \qquad \text{($\m I$ nodal interpolant)}
\]
even in the limit as the mesh is refined (see Example \ref{ex:inconsistency} below). This indicates that the orthogonality properties of the projection of the exact solution into the approximation space play an important role for the understanding of the (pointwise) consistency of the finite element scheme. Furthermore, this interpretation may serve as a starting point in selecting a discretisation of the HJB operator.

Our analysis combines the following key elements in a single finite element framework:

{\em Treatment of nodally inconsistent discretisations and uniform convergence:} The consistency condition (see \cite[eqn.(2.4)]{BS91} or \cite[p.332]{FS06}) of Barles and Souganidis is based on a limit involving pointwise values of smooth test functions. This condition is not satisfied by finite element methods, even for linear equations. Based on an alternative consistency condition we show the uniform convergence of finite element solutions to the viscosity solution.

{\em Gradient convergence:} We demonstrate how the coercivity of the linear operators under the supremum is recovered by the finite element method in order to control the gradient of the numerical solutions. In a uniformly elliptic setting, this leads  to strong convergence in $L^2([0,T],H^1(\O))$.

{\em Operators of non-negative characteristic form:} The presented analysis includes the treatment of partially and fully deterministic optimal control problems, corresponding to degenerate elliptic operators under the supremum of the Hamiltonian.

{\em Unstructured meshes:} In the spirit of finite element methods the computational domain may be triangulated with an unstructured mesh, allowing to capture complex domains more easily than in a finite difference setting. Typically, weaker conditions on the mesh than quasi-uniformity can be made.

{\em Regularisation with second order operators:} We highlight that the regularisation with second order elliptic operators is sufficient to achieve convergence to the viscosity solution. Indeed, in the example of the method of artificial diffusion, we illustrate how the regularisation in the second order fully non-linear case is of the same kind and order as for first order linear operators.

{\em Unconditional time step:} Our analysis permits explicit, semi-implicit and fully implicit discretisations in time. Fully implicit discretisations in time lead to unconditionally stable schemes.

The structure of the article is as follows: In Section II we introduce a framework of finite element methods. In Section III we study the well-posedness of the discrete systems of equations and describe how these systems are solvable by a known globally convergent, locally superlinearly convergent algorithm. Section IV establishes consistency properties of elliptic projection operators. This enables us to demonstrate in Section V that the upper and lower envelopes of the numerical solutions are sub- and supersolutions. Uniform convergence to the viscosity solution is derived in Section VI and is then built upon to analyse the convergence of the gradient in Section VII. We provide a concrete specimen of a scheme belonging to our framework by describing the method of artificial diffusion in Section VIII.

\section{Problem statement and definition of the numerical method} \label{sec:problemstatement}

Let $\O$ be a bounded Lipschitz domain in $\R^d$, $d \ge 2$. Let $A$ be a compact metric space and
\[
A \to C(\oO) \times C(\oO, \R^d) \times C(\oO) \times C(\oO),\; \alpha \to (a^\alpha,  b^\alpha, c^\alpha, d^\alpha)
\]
be continuous, such that the families of functions $\{a^{\a}\}_{\a\in A}$, $\{b^{\a} \}_{\a\in A}$, $\{c^{\a}\}_{\a\in A}$ and $\{d^{\a}\}_{\a\in A}$ are equi-continuous. Consider the bounded linear operators of non-negative characteristic form \cite{Oleinik}
\[
L^\alpha : \; H^1_0(\O) \to H^{-1}(\O), \; w \mapsto - a^\alpha \, \Delta w + b^\alpha \cdot \nabla w + c^\alpha \, w
\]
where $\alpha$ belongs to $A$. Furthermore, suppose that pointwise $d^\alpha \ge 0$. Then
\begin{align} \label{Lbounds}
\sup_{\alpha \in A} \| \, (a^\alpha,  b^\alpha, c^\alpha, d^\alpha) \,
\|_{C(\oO) \times C(\oO, \R^d) \times C(\oO) \times C(\oO)} < \infty, \qquad
\sup_{\alpha \in A} \| L^\alpha \|_{C^2(\oO) \to C(\oO)} < \infty.
\end{align}
We assume that the final-time boundary data \( v_T \in C(\oO) \) is non-negative: \( v_t \geq 0 \) on \(\oO\). For smooth $w$ let 
\[
H w := \sup_\alpha (L^\alpha w - d^\alpha),
\]
where the supremum is applied pointwise. The HJB equation considered is
\begin{subequations}
\begin{alignat}{2}
-\p_t v + H v &=0 	&	&\qquad\text{in }(0,T)\times\O,\\
v&=0 		  	&	&\qquad\text{on }(0,T)\times\pO,\\
v&=v_T		  	&	&\qquad\text{on }\{T\}\times\oO.
\end{alignat}
\end{subequations}
\begin{definition}[\cite{BP87,FS06}]
An upper semi-continuous (respectively lower semi-continuous) function \(v : [0,T] \times \oO \to \R \) is a viscosity subsolution (respectively supersolution) of
\begin{align} \label{HJB}
- \p_t v + H v = 0
\end{align}
on $(0,T) \times \O$ if for any \(w\in C^{\infty}(\R \times\R^d)\) such that \(v-w\) has a strict local maximum (respectively minimum) at \((t,x)\in(0,T)\times\O\) with \(v(t,x)=w(t,x)\), gives $-\p_t w(t,x)+Hw(t,x) \leq 0$, (respectively greater than or equal to \(0\)). If \(v\in C(\R\times\R^d)\) of equation \eqref{HJB} is a viscosity subsolution and supersolution, then \(v\) is called a viscosity solution.
\end{definition}

Let $V_i$ be a sequence of piecewise linear shape-regular finite element spaces with nodes $\il y$ and associated hat functions $\il \phi$. Here $\ell$ is the index ranging over the nodes of the finite element mesh. Let $V_i^0$ be the subspace of functions which satisfy homogeneous Dirichlet conditions. It is convenient to assume that \(\il y \in \O\) for $\ell \leq N := \dim V_i^0$; i.e.~the index $\ell$ first ranges over internal and then over external nodes. Set $\hil \phi := \il \phi / \| \il \phi \|_{L^1(\O)}$. The mesh size, i.e. the largest diameter of an element, is denoted $(\Delta x)_i$. It is assumed that $(\Delta x)_i \to 0$ as $i \to \infty$.

Let $h_i$ be the (uniform) time step used in conjunction with $V_i$, with $\frac{T}{h_i} \in \N$, and let $\ik s$ be the $k$th time step at the refinement level~$i$. The set of time steps is
\[
S_i := \bigl\{ \ik s : k = 0, \ldots, {\textstyle \frac{T}{h_i}} \bigr\}.
\]
Let the $\ell$th entry of $d_i w( \ik s, \cdot)$ be
\[
(d_i w( \ik s, \cdot))_\ell = \frac{w( \iko s, \il y) - w(\ik s, \il y)}{h_i}.
\]
For each $\alpha$ and $i$ find an approximate splitting $L^\alpha \approx \ia E + \ia I$ into linear operators
\begin{align*}
\ia E &: \; H^1_0(\O) \to H^{-1}(\O), \; w \mapsto - \bia a \, \Delta w + \bia b \cdot \nabla w + \bia c \, w,\\
\ia I &: \; H^1_0(\O) \to H^{-1}(\O), \; w \mapsto - \bbia a \, \Delta w + \bbia b \cdot \nabla w + \bbia c \, w,
\end{align*}
with continuous
\begin{align} \label{conta} \begin{array}{ll}
A \to C(\oO) \times C(\oO, \R^d) \times C(\oO),&\; \alpha \to (\bia a, \bia b, \bia c),\\[1mm]
A \to C(\oO) \times C(\oO, \R^d) \times C(\oO),&\; \alpha \to (\bbia a, \bbia b, \bbia c)
\end{array} \end{align}
such that \(\bia c\) and \(\bbia c\) are non-negative and for some $\gamma \in \R$ and all $\alpha \in A$,
\begin{align}\label{react}
\gamma \ge \| \bia c \|_{L^\infty} + \| \bbia c \|_{L^\infty}.
\end{align}
Also find for each $i$ a non-negative $\ia d$ which approximates $d^\alpha$: $\ia d \approx d^\alpha$. These consistency conditions $L^\alpha \approx \ia E + \ia I$ and $d^\alpha \approx \ia d$ are made precise as follows:

\begin{assumption} \label{consistency}
For all sequences of nodes $(\il y)_{i \in \N}$, where in general $\ell = \ell(i)$ depends on $i$:
\[
\lim_{i \to \infty} \sup_{\a\in A} \bigl\| a^\alpha - \bigl( \bia a(\il y) + \bbia a(\il y) \bigr) \bigr\|_{L^\infty({\rm supp} \, \hil \phi)} + \bigl\| b^\alpha - \bigl(
\bia b + \bbia b \bigr) \bigr\|_{L^\infty(\oO,\R^d)} + \bigl\| c^\alpha - \bigl( \bia c + \bbia c \bigr) \bigr\|_{L^\infty(\oO)} + \bigl\| d^\alpha - \ia d
\bigr\|_{L^\infty(\oO)} = 0.
\]
\end{assumption}

Define, for $w \in H^1(\O)$, \(\ell \in \{ 1, \dots, N = \dim V_i^0 \} \),
\begin{subequations}\label{eq:discreteop}
\begin{align}
(\sia E w)_\ell & := \bia  a(\il y) \langle \nabla w, \nabla \hil \phi \rangle + \langle \bia  b \cdot \nabla w + \bia  c \, w, \hil \phi \rangle,\\
(\sia I w)_{\ell} & := \bbia a(\il y) \langle \nabla w, \nabla \hil \phi \rangle + \langle \bbia b \cdot \nabla w + \bbia c w, \hil \phi \rangle,\\
(\sia C)_\ell & := \langle \ia d, \hil \phi \rangle.
\end{align}
\end{subequations}
On the restriction to $V_i$ we identify the $\sia E w$ and $\sia I w$ with their matrix representations with respect to the nodal basis \(\left\{\il \phi\right\}_{\ell}\). Similarly the nodal evaluation operator corresponds then to the identity matrix $\Id$.

\begin{definition}
An operator \( F : V \to \R^N \) is said to satisfy the Local Monotonicity Property (LMP) property if for all \(v \in V_i\) such that \(v\) has a non-positive minimum at the internal node \(\il y\), \(\ell\in \{1, \dots, N \}\), we have $( F v )_{\ell} \leq 0$. The operator $F$ satisfies the weak Discrete Maximum Principle (wDMP) provided that: 
\Eb\label{eq:wDMP}
\text{if \,}\bigl( F w \bigr)_{\ell} \geq 0 \text{ for all } \ell \in \{1,\dots, N\}, \quad\text{ then } \min_{\O} w \geq \min \{ \min_{\pO} w, 0 \}.
\Ee
\end{definition}

More explicit alternative formulations of the wDMP are discussed, for example, in \cite{BE02} and \cite{BE05}. Note that also $\Id$ and $\mathsf 0$ satisfy this LMP property. It is clear that if $F$ satisfies the LMP and \(v \in V_i\) has a {\em negative} minimum at the internal node \(\il y\) then $\bigl( (F + \eps \, \Id) v \bigr)_{\ell} < 0$ for all $\eps > 0$. This implies for all $\eps > 0$ that $F + \eps \, \Id$ satisfies the wDMP.

\begin{assumption} \label{monotonicity}
Assume for each $\alpha \in A$ that\, $\sia E$ restricted to $V_i$ has non-positive off-diagonal entries. Let $h_i$ be small enough so that all $h_i \sia E - \Id$ are monotone, i.e.~so that all entries of all $h_i \sia E - \Id$ are non-positive. Assume that for each \(\a\in A\) that \(\; \sia I\) satisfies the LMP property.
\end{assumption}

Obtain the numerical solution $v_i(T, \cdot) \in V_i$ by interpolation of $v_T$. Then $v_i(\ik s, \cdot) \in V_i^0$ at time $\ik s$ is defined, inductively, by
\begin{align} \label{num}
- d_i v_i(\ik s,\cdot) + \sup_\alpha \bigl( \sia E v_i(\iko s,\cdot) + \sia I v_i(\ik s,\cdot) - \sia C \bigr) = 0.
\end{align}
If all $\sia I$ are $\mathsf 0$ then \eqref{num} is an explicit scheme, otherwise implicit. Notice that the monotonicity assumption on $h_i \sia E - \Id$ is a time step restriction if $\sia E$ has positive diagonal entries.

\section{Well-posedness of the discrete HJB equations}

Let $\alpha_i^{\ell,k}(w)$ be a control $\alpha$ which maximises
\begin{align} \label{eq:max_alpha}
\sup_\alpha \left( \sia E w(\iko s,\cdot) + \sia I w(\ik s,\cdot) - \sia C \right)_{\ell} .
\end{align}
Let $\siwk I$ and $\siwk E$ be the matrices whose $\ell$th row is equal to that of
\[
\si I^{\alpha_i^{\ell,k}(w)} \quad \mbox{and} \quad \si E^{\alpha_i^{\ell,k}(w)},
\]
respectively. Also let the $\ell$th entry of $\siwk C$ be
\[
\si C^{\alpha_i^{\ell,k}(w)}.
\]
Thus, informally speaking, the $\siwk E$, $\siwk I$ and $\siwk C$ are gained by `reshuffling' the rows of the $\sia E$, $\sia I$ and $\sia C$, respectively.  Notice that the maximising control in \eqref{eq:max_alpha} may be non-unique.

Where no ambiguity can arise we simply write $\siw I$, $\siw E$ and $\siw C$ without explicitly referring to $k$. We will make use of the partial ordering of \(\RN\):
\eb
\text{for }x,\,y\in \RN;\quad x\geq y \text{ if and only if } x_{\ell} \geq y_{\ell},\quad \forall \ell\in\left\{1,\dots,n\right\}.
\ee
For a collection \(\left\{x^{\a}\right\}_{\a\in A}\subset\RN\), we define the operator \(\sup_{\a \in A}\) componentwise: $\bigl(\sup_{\a\in A} x^{\a}\bigr)_{\ell}=\sup_{\a\in A}x^{\a}_{\ell}$. The following lemma shows that in the linear case the wDMP turns, for functions which vanish on the boundary, into an M-matrix property.
 
\begin{lemma}\label{lem:monotonicity}
The matrices $h_i \siwk E - \Id$ are monotone. The matrices of\, $h_i \siwk I + \Id$ restricted to \(V_i^0\) are invertible diagonally dominant M-matrices for all \(w \in C([0,T]\times\oO)\).  The operators $\siwk I$ and $h_i \siwk I + \Id$ satisfy the LMP and wDMP, respectively.
\end{lemma}

\begin{proof}
Monotonicity of $h_i \siwk E - \Id$ is a straightforward consequence of the non-positivity of the entries of $h_i \sia E - \Id$ for all \(\a\in A\). The LMP property of \(\sia I\) for the node \(\il y\) only imposes a condition on the \(\ell\)-th row of the matrix of \(\sia I\). Hence it is easily checked that the $\siwk I$ and the \( h_i \siwk I +\Id\), which are composed row-wise from the $\sia I$ and \(h_i \sia I + \Id\), satisfy the LMP and wDMP respectively when all \(\sia I\) satisfy the LMP property.

The LMP property also implies that the matrix representations of the \( \sia I\) restricted to \(V_i^0\) are weakly diagonally dominant for all \(\a\in A\). This is because taking \(v = - \sum_{\ell=1}^N \il \phi\) yields
\eb
0 \geq \left(\sia I v\right)_{\ell}=-\left(\sia I \right)_{\ell \ell}-\sum_{j \neq \ell}^N \left(\sia I \right)_{\ell j},
\ee
using the fact that $v$ attains a non-positive minimum at each internal node.  For $j \neq \ell$ the hat function $\phi_i^j$ attains a non-positive minimum at $\il y$, giving \(\left(\sia I \right)_{\ell j}\leq 0\). This yields
\eb
\left(\sia I \right)_{\ell \ell} - \sum_{j\neq \ell}^N \abs{\left(\sia I \right)_{\ell j}} \geq 0.
\ee
Therefore \(h_i \siwk I + \Id\) restricted to \(V_i^0\) is strictly diagonally dominant and thus invertible. Furthermore, since $(h_i \siwk I + \Id ) + \eps \, \Id$ is similarly invertible for all \(\eps \geq 0\) and all off-diagonal entries are non-positive, \cite[p.\ 114]{HJ91} shows that \(h_i \siwk I + \Id\) restricted to \(V_i^0\) is represented by an invertible M-matrix.
\end{proof}

\begin{corollary}
The non-linear operators $w \mapsto \siwk I w$ and \( w \mapsto (h_i \siwk I + \Id)w \) satisfy the LMP and wDMP, respectively. Moreover, \( w \mapsto -(h_i \siwk E - \Id) w \) is positive: if $w \ge 0$ then $-(h_i \siwk E - \Id) w \ge 0$.
\end{corollary}

We record a constructive proof of existence of a solution \(v_i \in S_i\times V_i^0\) to \eqref{num} for all \(k \in \left\{0,1,2,\dots,T/h_i-1\right\}\) with the below Algorithm~\ref{alg}. This algorithm, which can be traced back to \cite{Howard60}, is found in the continuous setting in \cite{LM80} which provides the proof of convergence and existence of solutions. In \cite{BMZ09} it is shown that in the discrete setting it is a semi-smooth Newton method that converges superlinearly.

The algorithm to solve the non-linear problem \eqref{num} at a given time level is the following.

\begin{algorithm}\label{alg}
Given \(k \in \N \) and \(v_i(s_i^{k+1},\cdot) \in V_i^0\), choose an arbitrary \(\a \in A\) and find \(w_0 \in V_i^0\) such that
\eb
\left(h_i \sia I + \Id\right) w_0 = h_i \sia C- \left(h_i \sia E - \Id\right)v_i(s_i^{k+1},\cdot).
\ee
For \(m\in \left\{0,1,2,\dots\right\}\), inductively find \(w_{m+1} \in V_i^0\) such that
\Eb\label{algeq}
\left(h_i \mathsf{I}^{w_m}_i+\Id\right)w_{m+1}= h_i \mathsf{C}^{w_m}_i- \left(h_i \mathsf{E}^{w_m}_i-\Id\right)v_i(s_i^{k+1},\cdot).
\Ee
\end{algorithm}

\begin{theorem}\label{thm:discretewellposedness}
The numerical solution \(v_i\) exists, is unique, solves the linear systems
\Eb \label{numsol}
(h_i \si I^{k,v_i} + \Id) v_i(\ik s,\cdot) = - (h_i \si E^{k,v_i} - \Id) v_i(\iko s,\cdot) + h_i \si C^{v_i} \quad \tst \forall k \in \bigl\{0,1,2,\dots,\frac{T}{h_i}-1 \bigr\};
\Ee
and is non-negative. Given \(k\in\left\{0,1,2,\dots\right\}\) and \(v_i(s_i^{k+1},\cdot) \in V_i^0\) , the iterates of Algorithm~\ref{alg} converge superlinearly to the unique solution \(v_i(s_i^k,\cdot)\) of \eqref{num}: \(\lim_m w_m = v_i(\ik s,\cdot)\). Any numerical solution $\ia v$ of the linear evolution problem associated to a fixed $\alpha$ with homogeneous Dirichlet conditions, that is
\begin{align} \label{numlinearsol} \begin{array}{c}
(h_i \sia I + \Id) \ia v(\ik s,\cdot) = - (h_i \sia E - \Id) \ia v(\iko s,\cdot) + h_i \sia C, k \in \{1, 2, 3, \ldots \}, \quad \mbox{with} \quad
\ia v(T,\cdot) = v_i(T,\cdot),\\[1mm] v_i(s_i^k,\cdot) \in V_i^0\quad \text{for all } \tst k\in\bigl\{0,1,\dots,\frac{T}{h_i}-1\bigr\};
\end{array} \end{align}
is an upper bound: $v_i \le \ia v$ on \(S_i\times\oO\).
\end{theorem}

\begin{proof} \cite[Theorem 2.1]{BMZ09} shows existence and uniqueness of a solution \(v_i(s_i^k,\cdot)\) given \(k\) and \(v_i(s_i^{k+1},\cdot)\) and superlinear convergence of the algorithm: their Assumption (H1) is ensured by Lemma \ref{lem:monotonicity} and their Assumption (H2) is guaranteed by equation \eqref{conta}. Existence of a solution \(v_i\) is then obtained by induction over \(k\).

Also \(v_i \geq 0\) on \(S_i\times\oO\) follows from induction over \(k\). By assumption \(v_i(T,\cdot) \geq 0 \) on \(\oO\). Since all entries of \( h_i \si{E}^{v_i} - \Id \) are non-positive, all entries of \( \si{C}^{v_i}\) are non-negative, and \(v_i(s_i^{k+1},\cdot)\geq0\), \eqref{numsol} shows
\begin{align*}
(h_i \si{I}^{v_i} + \Id) v_i(\ik s,\cdot) & = - (h_i \si{E}^{v_i} - \Id) v_i(\iko s,\cdot) + h_i \si C^{v_i} \geq 0.
\end{align*}
Hence by inverse positivity of \( h_i \si{I}^{v_i} + \Id\), we deduce that \( v_i(s_i^k,\cdot)\geq 0\) on \(\oO\).

Finally, we prove that \(v_i \leq v_i^{\a}\) for all \(\a \in A\). Fix \(\a \in A\). Firstly, \(v_i(T,\cdot)=v_i^{\a}(T,\cdot)\). For given \( k \in \N \) assume that \(v_i(s_i^{k+1},\cdot) \leq v_i^{\a}(s_i^{k+1},\cdot)\). Then \eqref{num} implies
\eb
\left(h_i \sia I+\Id\right) v_i(s_i^k,\cdot) \leq h_i \sia C -  \left(h_i\sia E - \Id\right)v_i(s_i^{k+1},\cdot).
\ee
Subtracting \eqref{numlinearsol} from the above inequality and using monotonicity of \( h_i \sia E-\Id\) yields
\begin{align*}
\left(h_i \sia I+\Id\right)\left(v_i(s_i^k,\cdot)-v_i^{\a}(s_i^k,\cdot)\right) &\leq \left(h_i \sia E - \Id\right)\left( v_i^{\a}(s_i^{k+1},\cdot)-v_i(s_i^{k+1},\cdot)\right) \leq 0.
\end{align*}
Thus by inverse positivity of \( h_i \sia I + \Id \) we conclude that \( v_i(s_i^{k},\cdot) \leq v_i^{\a}(s_i^{k},\cdot)\) on \(\O\), which completes the induction.
\end{proof}

\section{Consistency properties of elliptic projections}

The Barles-Souganidis argument requires the existence of a projection operator onto the discrete function space that satisfies two properties. First, the projections of a smooth function must be convergent in a sufficiently strong sense, for example in $W^{1,\infty}$. Second, the discretisations of the partial differential operators must be pointwise consistent when applied to the projections of a smooth function, \ie the values of the operators applied to the projections converge to the values of the continuous operator applied to the smooth function. In the context of classical finite difference methods, the interpolant to the grid satisfies these properties trivially because the operators are designed to be consistent with respect to interpolation. However, in the case of FEM, the nodal interpolant may fail to satisfy the consistency condition, even for reasonable meshes. We illustrate this behaviour in Example~\ref{ex:inconsistency}.

\begin{figure}[t]
\begin{center}
 \subfloat[Consistent mesh]{
	  \begin{tikzpicture}[scale=1.5,auto]
	  \fill[black] (0,0) circle (0.05) node[above left] {$x$};
	  \draw (-1,0) node[left] {$y^4_i$} --  (1,0) node[right] {$y_i^2$};
	  \draw (0,-1) node[below] {$y_i^5$} --(0,1) node[above] {$y_i^3$} ;
	  \draw (1,0)  -- (1,1);
	  \draw (1,1)  --  (0,1) ;
	  \draw (0,1) --  (-1,0);
	  \draw (-1,0)  --  (-1,-1);
	  \draw (-1,-1) --  (0,-1);
	  \draw (-1,-1) -- (1,1);
	  \draw (0,-1)  -- (1,0);
	  \end{tikzpicture}\label{fig:consmesh}
	  }
\hspace{2cm}
  \subfloat[Inconsistent mesh]{
	  \begin{tikzpicture}[scale=1.5,auto]
	  \fill[black] (0,0) circle (0.05) node[above left] {$x$};
	  \node at (0.25,0.25) {$T_1$};
	  \draw (-1,0) node[left] {$y_i^4$} --  (1,0) node[right] {$y_i^2$};
	  \draw (0,-1) node[below] {$y_i^5$} --(0,1) node[above] {$y_i^3$} ;
	  \draw (1,0)--(0,1);
	  \draw (0,1)--(-1,0);
	  \draw (-1,0) -- (0,-1);
	  \draw (0,-1)--(1,0);
	  \end{tikzpicture}\label{fig:inconsmesh}
	  }
\end{center}
\caption{(a) illustrates a mesh that leads to a FEM discretisation of the Laplacian that is pointwise consistent with respect to the interpolant. This is no longer the case for the mesh depicted by (b). In (b), \(T_1\) denotes the upper-right element.}
\label{fig:meshes}
\end{figure}
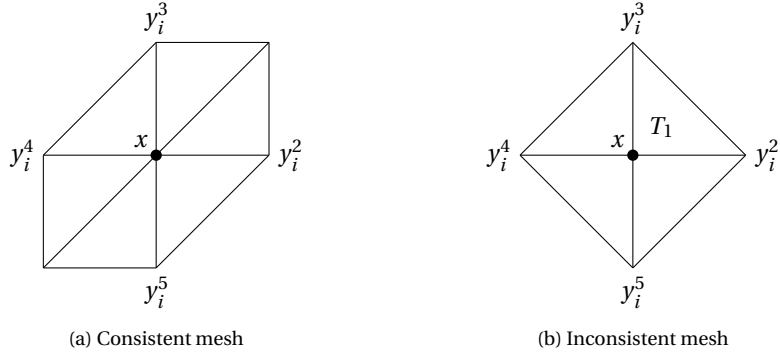

\begin{example}\label{ex:inconsistency}
For a fixed point \(x\) in a domain, consider two sequences of meshes, such that the elements neighbouring \(x\) are as depicted in Figure \ref{fig:meshes}. Denote \(\hat{\phi}_i\) and \(\hat{\varphi}_i\) the \(L^1\)-normalised hat functions associated with the node \(x\) for the meshes depicted respectively by (a) and (b). Let \(w\) be a smooth function; let \(\calI_a w\) and \(\calI_b w\) be the nodal interpolants of \(w\) respectively on the two meshes. We show that the mesh type of (a) leads to a FEM discretisation of the Laplacian that is strongly consistent with respect to interpolation, whereas the mesh type of (b) does not.

For the mesh of Figure \ref{fig:meshes}(a), it is well known that the FEM discretisation of the Laplacian coincides with a finite difference discretisation and that
\begin{align*}
\pair{\nabla \calI_a w}{\nabla \hat{\phi}_i} &=\frac{1}{(\Delta x)_i^2}\left(4w(x)-w(y^2_i)-w(y^3_i)-w(y^4_i)-w(y^5_i)\right) = - \Delta w(x)+ O((\Delta x)_i^2).
\end{align*}
For the mesh of Figure \ref{fig:meshes}(b), we sketch the calculation: first we have
\eb
\norm{\varphi_i}=\frac{2}{3} (\Delta x)_i^2;
\quad \eval{\nabla \hat{\varphi}_i}{T_1}=\frac{3}{2 (\Delta x)_i^3} \begin{pmatrix} -1 \\ -1 \end{pmatrix};
\quad \eval{\nabla \calI_b w}{T_1}=\frac{1}{(\Delta x)_i} \begin{pmatrix} w(y^3_i)-w(x)\\ w(y^2_i)-w(x) \end{pmatrix};
\ee
thus
\eb
\int\limits_{T_1} \nabla \calI_b w \cdot \nabla \hat{\varphi}_i \, \d x = \frac{3}{4 (\Delta x)_i^2} \left( 2w(x)- w(y^3_i) - w(y^2_i) \right).
\ee
Doing a similar calculation for the other elements shows that
\begin{align*}
 \pair{\nabla \calI_b w}{\nabla \hat{\varphi}_i}
& = \frac{3}{2 (\Delta x)_i^2} \left(4w(x)-w(y^2_i)-w(y^3_i)-w(y^4_i)-w(y^5_i)\right)
 = - \frac{3}{2} \Delta w(x) + O( (\Delta x)_i^2).
\end{align*}
\end{example}

We overcome this difficulty by using a different projection operator in the Barles-Souganidis argument. Given $w \in C([0,T], H^1(\O))$, denote by $\P w$ a linear mapping into $[0,T] \times V_i$ which satisfies for all $\hil \phi \in V_i^0$
\begin{equation}\label{eq:ellprojdef}
\langle \nabla \P w(t,\cdot), \nabla \hil \phi \rangle = \langle \nabla w(t,\cdot), \nabla \hil \phi \rangle \quad \forall t \in [0,T].
\end{equation}
Notice that $\P$ coincides with the classical elliptic projection of the Laplacian if $\P w$ is chosen to interpolate $w$ on the boundary.
\begin{assumption}\label{ass:ellproj}
There are mappings \(\P\) satisfying \eqref{eq:ellprojdef} and there is a constant \(C\geq 0\) such that for every \( w \in C^{\infty}(\R^{d})\) and \(i\in\N\),
\Eb\label{ass:ellprojstab}
\norm{\P w}_{W^{1,\infty}(\O)} \leq C \norm{w}_{W^{1,\infty}(\O)}
\quad\text{and}\quad
\lim_{i\tends \infty} \norm{\P w - w}_{W^{1,\infty}(\O)}=0.
\Ee
\end{assumption}

The conditions under which the above assumption holds for the elliptic projection typically include a condition on the mesh grading and on the domain. In \cite{DLSW}, it is shown that \eqref{ass:ellproj} holds when \(\O\) is a bounded convex polyhedral domain in \(\R^d\), \(d \in \{2,3\} \), when the mesh satisfies a local quasi-uniformity condition and when the test functions vanish on the boundary. To apply the result for non-convex domains $\O$ and general \(w\in C^{\infty}(\R \times \R^d)\), consider for example a convex polyhedral domain $B$ containing $\O$  and assume there is a locally quasi-uniform mesh on $B$ which coincides with the original mesh on $\O$. Let \(\eta\) be a smooth cut-off function with compact support in \(B\) such that \(\eta \equiv 1\) on \(\O\). Then the classical elliptic projection on $B$, acting on $\eta w : B \to \R$, has the required properties. Given this construction for \(\P\), it is natural to refer to it as an elliptic projection.

\begin{lemma}\label{lem:ellconsistency1}
Let \(w\in C^{\infty}(\R \times \R^d)\) and let \( \bigl\{s_i^{k(i)}\bigr\}_{i=1}^{\infty} \) tend to \(t \in [0,T)\). Then
\Eb\label{eq:timederivconv}
\lim_{i\tends \infty} d_i \P w(s_i^{k(i)},\cdot) = \p_t w(t,\cdot) \text{ in } W^{1,\infty}(\O).
\Ee
\end{lemma}

\begin{proof}
By linearity of $\P$ and \eqref{ass:ellprojstab}, the result follows in the limit $i \to \infty$ from
\begin{align*}
& \| d_i \P w(s_i^{k(i)},\cdot) -\p_t w(t,\cdot) \|_{W^{1,\infty}(\O)}\\
\le \, & \| d_i \P w(s_i^{k(i)},\cdot) - d_i \P w(t,\cdot) \|_{W^{1,\infty}(\O)} + \| d_i \P w(t,\cdot) - P_i \p_t w(t,\cdot) \|_{W^{1,\infty}(\O)} + \| P_i \p_t w(t,\cdot) -\p_t w(t,\cdot) \|_{W^{1,\infty}(\O)}\\
\le \, & C \| d_i w(s_i^{k(i)},\cdot)  - d_i w(t,\cdot) \|_{W^{1,\infty}(\O)} + C \| d_i w(t,\cdot) - \p_t w(t,\cdot) \|_{W^{1,\infty}(\O)} + \| P_i \p_t w(t,\cdot) -\p_t w(t,\cdot) \|_{W^{1,\infty}(\O)},
\end{align*}
where $d_i \P w(t,\cdot) = (\P w(t + h_i,\cdot) - \P w(t,\cdot) ) / h_i$, assuming $i$ is sufficiently large to ensure $t + h_i \le T$.
\end{proof}

\begin{lemma}\label{lem:ellconsistency2}
Let \(w\in C^{\infty}(\R \times \R^d)\) and let \( \bigl\{s_i^{k(i)}\bigr\}_{i=1}^{\infty} \) tend to \(t\in [0,T]\), \( \bigl\{y_i^{\ell(i)}\bigr\}_{i=1}^{\infty}\) tend to \( x \in \O\). Then
\Eb\label{eq:ellconsistency}
\lim_{i\tends \infty} \left( \sia E \P w(s_i^{k(i)+1},\cdot) + \sia I \P w(s_i^{k(i)},\cdot)  - \sia C\right)_{\ell(i)} = L^{\a} w(t,x) - d^{\a}(x) \quad \text{uniformly over all }\a \in A.
\Ee
\end{lemma}

\begin{proof}
For ease of notation, the dependence of \(k\) and \(\ell\) on \(i\) is made implicit. From the definition of $\P$ and integration by parts,
\begin{align*}
& \abs{ \bbia a (\il y) \pair{\nabla \P w(s_i^{k},\cdot)}{\nabla \hil \phi} + \bia a(\il y) \pair{\nabla \P w(s_i^{k+1},\cdot)}{\nabla \hil \phi} - a^{\a} (\il y) \pair{ \nabla w(t,\cdot)}{\nabla \hil \phi}}\\
= \, & \abs{ \bbia a (\il y) \pair{\nabla  w(s_i^{k},\cdot)}{\nabla \hil \phi} + \bia a(\il y) \pair{\nabla  w(s_i^{k+1},\cdot)}{\nabla \hil \phi} - a^{\a} (\il y) \pair{ \nabla w(t,\cdot)}{\nabla \hil \phi}} \\
\leq \, & \abs{ \left(a^{\a} (\il y) - \bbia a  (\il y) - \bia a  (\il y)\right) \pair{-\Delta w(t,\cdot)}{\hil \phi}}
 + \abs{\bbia a (\il y) \pair{ \Delta w(t,\cdot) - \Delta w(\ik s,\cdot)}{\hil \phi}}
 + \abs{\bia a (\il y) \pair{ \Delta w(t,\cdot)-\Delta w (s_i^{k+1},\cdot)}{\hil \phi}}.
\end{align*}
Using Assumption \ref{consistency} and the continuity of \(w\colon [0,T]\mapsto C^2(\oO)\) together with uniform boundedness of \( \bigl\{\abs{\bbia a (\il y)}\bigr\}_{\a\in A}\) and \(\bigl\{\abs{\bia a (\il y)}\bigr\}_{\a\in A} \), we conclude that
\eb
\lim_{i\tends \infty} \sup_{\a\in A} \abs{ \bbia a (\il y) \pair{\nabla \P w(s_i^{k},\cdot)}{\nabla \hil \phi} + \bia a(\il y) \pair{\nabla \P w(s_i^{k+1},\cdot)}{\nabla \hil \phi} - a^{\a} (\il y) \pair{ \nabla w(t,\cdot)}{\nabla \hil \phi}} =0.
\ee
Owing to the Heine-Cantor theorem for all \(\eps >0\), there is a \(\delta>0\) such that $| \Delta w(t,x) - \Delta w(t,y) | < \eps$ if $|x-y| < \delta$. Since, for \(i\) sufficiently large, the support of $\hil \phi$ is contained in the ball $B(x,\delta)$ and since \( \norm{ \hil \phi }_{L^1(\O)}=1\) as well as \(\hil \phi \geq 0\), we find
\eb
\abs{ \Delta w(t,x) - \pair{\Delta w(t,\cdot)}{\hil \phi}} < \eps.
\ee
As \( \left\{a^{\a}\right\}_{\a \in A}\) is an equi-continuous family of functions, we conclude that
\eb
\lim_{i\tends \infty} \sup_{\a \in A} \abs{ a^{\a} (\il y) \pair{\Delta w(t,\cdot)}{\hil \phi} - a^{\a}(x)\Delta w(t,x)} = 0;
\ee
thus showing that
\Eb\label{eq:ellconsist1}
\lim_{i\tends \infty} \sup_{\a\in A} \abs{ \bbia a (\il y) \pair{\nabla \P w(s_i^{k},\cdot)}{\nabla \hil \phi} + \bia a(\il y) \pair{\nabla \P w(s_i^{k+1},\cdot)}{\nabla \hil \phi} - \left(- a^{\a}(x)\Delta w(t,x)\right)} =0.
\Ee
Using Assumption \ref{ass:ellproj} and regularity of \(w\), we see that \( \P w(s_i^{k},\cdot) \) and \( \P w(s_i^{k+1},\cdot)\) converge to \(w(t,\cdot)\) in \(W^{1,\infty}(\O)\). It can then be shown by analogous estimates and by using the equi-continuity of \( \left\{b^{\a}\right\}_{\a\in A}\), \(\left\{c^{\a}\right\}_{\a\in A}\) and \(\left\{d^{\a}\right\}_{\a\in A}\), that
\begin{subequations} \label{eq:ellconsist_comb}
\begin{gather}
\lim_{i\tends \infty} \sup_{\a \in A} \abs{ \pair{\bbia b \cdot \nabla \P w(s_i^k,\cdot)}{\hil \phi}+\pair{\bia b \cdot \nabla \P w(s_i^{k+1},\cdot)}{\hil \phi} - b^{\a}(x)\cdot \nabla w(t,x) } =0, \label{eq:ellconsist2} \\
\lim_{i\tends \infty} \sup_{\a \in A} \abs{ \pair{\bbia c \P w(s_i^k,\cdot)}{\hil \phi} + \pair{\bia c \P w(s_i^{k+1},\cdot)}{\hil \phi}- c^{\a}(x)w(t,x)}=0, \label{eq:ellconsist3}\\
\lim_{i\tends \infty} \sup_{\a\in A} \abs{ \pair{d_i^{\a}}{\hil \phi} - d^{\a}(x)} =0.\label{eq:ellconsist4}
\end{gather}
\end{subequations}
Combining equations \eqref{eq:ellconsist1} and \eqref{eq:ellconsist_comb} yields \eqref{eq:ellconsistency}.\end{proof}

\section{Sub- and supersolution}

Set
\[
v^*(t,x) = \sup_{(\ik s,\il y) \to (t,x)} \limsup_{i \to \infty} v_i(\ik s,\il y), \qquad v_*(t,x) = \inf_{(\ik s,\il y) \to (t,x)} \liminf_{i \to \infty} v_i(\ik s,\il y)
\]
where the limit superior and limit inferior are taken over all sequences of nodes in $[0,T] \times \oO$ which converge to $(t,x) \in [0,T] \times \oO$. By construction, $v^*$ is upper and $v_*$ lower semi-continuous. With the use of elliptic projection operators key steps of the convergence proof in \cite{BS91}, which is stated there in a suitable form for finite difference methods, are transferred to finite element schemes, which do not satisfy the consistency condition in \cite{BS91}.

\begin{theorem}
The function $v^*$ is a viscosity subsolution of \eqref{HJB} and $v_*$ is a viscosity supersolution of \eqref{HJB}.
\end{theorem}

\begin{proof}
{\em Step 1 ($v^*$ is a subsolution).}  To show that $v^*$ is a viscosity subsolution, suppose that \(w \in C^{\infty}(\R \times \R^d)\) is a test function such that \(v^*-w\) has a strict local maximum at \( (s,y) \in (0,T)\times\O\), with \( v^*(s,y)=w(s,y)\). Consider a neighbourhood $B := \bigl\{ (t,x) \in (0,T)\times \O\; : \; |t-s|+|x-y| \leq \delta \bigr\}$ with \(\delta>0\) such that
\eb
v^*(s,y)-w(s,y) > v^*(t,x)-w(t,x) \quad \forall (t,x) \in B \setminus (s,y).
\ee
Choose $i$ sufficiently large for $B$ to contain nodes. Let \((s_i^k,y_i^\ell)\) denote the position where \( v_i(\ika s,\ila y) - \P w(\ika s,\ila y)\) attains the maximum among all nodes \((\ika s,\ila y) \in B \). Let us pass to a subsequence \(\bigl\{ (s_{i(j)}^k,y_{i(j)}^{\ell}) \bigr\}_{j\in\N}\) of \( \bigl\{(\ik s,\il y)\bigr\}_{i\in\N}\) for which \(\bigl\{v_i(s_{i(j)}^k,y_{i(j)}^{\ell})\bigr\}_{j\in\N}\) converges to the limit superior of \( \bigl\{v_i(\ik s,\il y)\bigr\}_{i\in\N}\). By compactness of \( B \), there is a subsequence of \(\bigl\{(s_{i(j)}^k,y_{i(j)}^{\ell})\bigr\}_{j\in\N}\) converging to a point \((\tilde{s} ,\tilde{y}) \in B \). Then \(\P w(s_{i(j)}^k,y_{i(j)}^{\ell}) \tends w(\tilde{s},\tilde{y})\) from \eqref{ass:ellprojstab} and by continuity of \(w\). As the \( (\ik s ,\il y) \) are maximisers, one has
\eb
v^*(\tilde{s},\tilde{y})-w(\tilde{s},\tilde{y}) = \limsup_{i \tends \infty} v_i(s_{i(j)}^k, y_{i(j)}^k)-\P w(s_{i(j)}^k, y_{i(j)}^k) = v^*(s,y)-w(s,y);
\ee
hence \( (\tilde{s},\tilde{y}) = (s,y) \) since \((s,y)\) is a strict maximiser of \(v^*-w\) on \( B \). Thus there is a subsequence of maximising nodes converging to \((s,y)\) to which we now pass without change of notation: \((\ik s,\il y) \tends (s,y)\). It follows that
\begin{align} \label{eq:defmu}
v_i(\ik s,\il y)-\P w(\ik s,\il y) \tends v^*(s,y)-w(s,y)=0.
\end{align}
Moreover, because of \((\ik s,\il y) \tends (s,y)\), the neighbours of the $(\ik s,\il y)$ eventually also belong to $B$: For \(i\) sufficiently large we have \(
(\ika s,\ila y) \in B \) if \(\kappa \in \left\{k,k+1\right\}\) and \(\ila y \in \supp \hil \phi\); in which case
\eb
v_i(\ika s,\ila y) - \P w (\ika s,\ila y) \leq v_i(\ik s,\il y) - \P w (\ik s,\il y) \qquad \Leftrightarrow \qquad \P w (\ika s,\ila y) + \mu_i \geq v_i  (\ika s,\ila y),
\ee
with \(\mu_i =  v_i(\ik s,\il y) - \P w (\ik s,\il y)\). Notice that $\mu_i \to 0$ as $i \to \infty$ because of \eqref{eq:defmu}.

Since the matrices \(\sia E\) have non-zero off diagonal entries \( \left(\sia E\right)_{\ell \lambda}\) only if \( \ila y \in \supp \hil \phi\), we have for all \(\a \in A\)
\eb
\left( (h_i \sia E - \Id ) \left[ \P w (s^{k+1},\cdot)+\mu_i \right] \right)_{\ell} \leq \left( (h_i \sia E-\Id) v_i(s_i^{k+1},\cdot)\right)_{\ell}.
\ee
By the LMP property and linearity of \( \sia I \), since \( \P w (\ik s,\cdot) +\mu_i - v_i(\ik s,\cdot)\) has a non-positive minimum at \( \il y\),
\eb
\left( (h_i \sia I + \Id ) \left[\P w(\ik s,\cdot)+\mu_i \right]\right)_{\ell} \leq \left( (h_i \sia I + \Id) v_i(\ik s,\cdot)\right)_{\ell}.
\ee
From the definition of the scheme,
\begin{subequations}
\begin{align}
0 &= - d_i v_i (\ik s,\il y) + \sup_{\a \in A} \left( \sia E v_i(s_i^{k+1},\cdot)+\sia I v_i(\ik s,\cdot) - \sia C \right)_{\ell}  \\
 & \geq - d_i \left( \P w (\ik s,\il y) + \mu_i \right) + \sup_{\a\in A} \left( \sia E \left( \P w (s_i^{k+1} ,\cdot) +  \mu_i \right) + \sia I \left( \P w (\ik s,\cdot) + \mu_i \right) - \sia C \right)_{\ell}\\
& = - d_i \P w(\ik s, \il y) + \sup_{\a \in A} \left[ \left( \sia E \P w (s_i^{k+1} ,\cdot) + \sia I\P w (\ik s,\cdot) - \sia C \right)_{\ell} + \mu_i \pair{\bia c+\bbia c}{\hil \phi}\right] \\
& \geq  - d_i \P w(\ik s, \il y) + \sup_{\a \in A} \left( \sia E \P w (s_i^{k+1} ,\cdot) + \sia I\P w (\ik s,\cdot) - \sia C \right)_{\ell} - \gamma \abs{\mu_i}.\label{eq:subsolineq1}
\end{align}
\end{subequations}
Since
\begin{align*}
& \abs{\sup_{\a\in A} \left( \sia E \P w (s_i^{k+1} s,\cdot) + \sia I\P w (\ik s,\cdot) - \sia C \right)_{\ell} - \sup_{\a\in A} \left( L^{\a} w(s,y)-d^{a}(y)\right)} \\
\leq \, & \sup_{\a\in A} \abs{  \left( \sia E \P w (s_i^{k+1} s,\cdot) + \sia I\P w (\ik s,\cdot) - \sia C \right)_{\ell} - \left(  L^{\a} w(s,y)-d^{a}(y) \right) },
\end{align*}
Lemmas \ref{lem:ellconsistency1} and \ref{lem:ellconsistency2} show that after taking the limit \(i\tends \infty\) in inequality \eqref{eq:subsolineq1} and recalling that \(\mu_i \tends 0\), we obtain
\Eb
0 \geq - \p_t w(s,y)+\sup_{\a\in A} \left(L^{\a} w(s,y)-d^{\a}(y)\right).
\Ee
Therefore \(v^*\) is a viscosity subsolution.

{\em Step 2 ($v_*$ is a supersolution).} Arguments similar to those above show that \(v_*\) is a viscosity supersolution, where the principal changes are that one considers \(w \in C^{\infty}(\R \times \R^d)\) such that \(v_*-w\) has a strict local minimum at some \((s,y)\in(0,T)\times\O\) with \(v_*(s,y)=w(s,y)\). Using analogous notation, inequality \eqref{eq:subsolineq1} corresponds to
\eb
0 \leq -d_i \P w(\ik s,\il y)+\sup_{\a\in A}\left(\sia E\P w(s_i^{k+1},\cdot)+\sia I \P w(\ik s,\cdot)-\sia C\right)_{\ell}+ \gamma \abs{\mu_i},
\ee
i.e. there is a slight asymmetry in the argument due to the last sign in \eqref{eq:subsolineq1}. Nevertheless it is then deduced that
\eb
0 \leq -\p_t w(s,y)+\sup_{\a \in A}\left(L^{\a} w(s,y)-d^{\a}(y)\right).
\ee
Thus \(v_*\) is a viscosity supersolution.
\end{proof}

\section{Uniform convergence}

We now turn to the initial and boundary conditions. Together with the sub- and supersolution property we appeal to a comparison principle to obtain uniform convergence of the numerical solutions.

For each \(\a \in A\), define
\eb
v^{\a,*}(t,x) = \sup_{(\ik s,\il y) \to (t,x)} \limsup_{i \to \infty} v_i^{\a}(\ik s,\il y);
\ee
where the $v_i^\alpha$ are as in \eqref{numlinearsol} and the limit superior is taken over all sequences of nodes which converge to $(t,x) \in [0,T] \times \oO$.
\begin{assumption}\label{ass:boundaryconv}
Suppose that for each \((t,x)\in[0,T]\times\pO\)
\begin{equation}\label{eq:boundaryconv}
\inf_{\a\in A} v^{\a,*}(t,x)=0.
\end{equation}
\end{assumption}

Before further considerations, let us motivate Assumption~\ref{ass:boundaryconv} with a simple example. As a side remark, this example also illustrates how in some settings Kushner-Dupuis finite difference schemes, as described in \cite{KD01,FS06}, may be interpreted as finite element methods in the framework of this paper.

\begin{example}\label{ex:hjbboundary}
Consider the backward time-dependent equation in one spatial dimension
\Eb\label{eq:hjbexampleeq}
-v_t+\abs{v_x}=1 \quad \text{on } (0,1)\times(-1,1),
\Ee
with boundary conditions \(v=0\) on \([0,1]\times\left\{-1,1\right\}\cup\left\{1\right\}\times[-1,1]\).  Equation \eqref{eq:hjbexampleeq} may be re-written in HJB form as
\eb
-v_t+\sup_{\a \in \left\{-1,1\right\}}\left(\a u_x-1\right)=0.
\ee
The viscosity solution is \( v=\min(1-t,1-\abs{x})\). We choose a uniform mesh with element size \(2 (\Delta x)_i\) and we use a fully explicit discretisation, where monotonicity will be achieved by using the method of artificial diffusion, as described in \cite{BE02}. Thus we have
\eb
\left(\sia E w\right)_{\ell} = \eps \pair{\p_xw}{\p_x \hil \phi} + \alpha \pair{\p_x w}{\hil \phi},
\ee
where \(\eps\) is the artificial diffusion parameter to be chosen to obtain a monotone scheme.  Calculating the entries shows that the \(\sia E\) are of the form
\eb
\eps
\begin{pmatrix}
\frac{2}{(\Delta x)_i^2} & - \frac{1}{(\Delta x)_i^2} &  &  & \\
 \ddots & \ddots & \ddots &  &   \\
 & -\frac{1}{(\Delta x)_i^2} & \frac{2}{(\Delta x)_i^2} &
-\frac{1}{(\Delta x)_i^2} &  \\
  &  & \ddots & \ddots & \ddots \\
 &  &  & - \frac{1}{(\Delta x)_i^2} & \frac{2}{(\Delta x)_i^2}
\end{pmatrix}
+ \a
\begin{pmatrix}
0 & \frac{1}{2 (\Delta x)_i} &  &  & \\
\ddots & \ddots & \ddots &   &  \\
 & -\frac{1}{2 (\Delta x)_i} & 0 & \frac{1}{2 (\Delta x)_i} &  \\
  &  & \ddots & \ddots & \ddots \\
 &  &  & -\frac{1}{2 (\Delta x)_i} & 0\\
\end{pmatrix}.
\ee
For monotonicity we require that all off-diagonal terms of the \(\sia E\) be non-positive, \ie we require \(\eps \geq (\Delta x)_i/2\). For example the special choice \(\eps=(\Delta x)_i/2\) yields
\eb
\mathsf{E}^1_i=
\begin{pmatrix}
\frac{1}{(\Delta x)_i} & 0 &  &  & \\
\ddots & \ddots &  &  &   \\
 & -\frac{1}{(\Delta x)_i} & \frac{1}{(\Delta x)_i} & 0 & \\
 &  & \ddots & \ddots & \\
&  &  & -\frac{1}{(\Delta x)_i} & \frac{1}{(\Delta x)_i}\\
\end{pmatrix};
\quad
\mathsf{E}^{-1}_i=
\begin{pmatrix}
\frac{1}{(\Delta x)_i} & -\frac{1}{(\Delta x)_i} & &  & \\
 & \ddots & \ddots &  &   \\
 & 0 & \frac{1}{(\Delta x)_i} & -\frac{1}{(\Delta x)_i} &  \\
 &  &  & \ddots & \ddots\\
 &  &  & 0 & \frac{1}{(\Delta x)_i}\\
\end{pmatrix}.
\ee
This is equivalent to discretising the spatial part of \(-v_t+v_x\) with backward finite differences and discretising the spatial part of \(-v_t-v_x\) with forward finite differences, as can be done in applying a Kushner-Dupuis scheme. It can then be deduced, whilst using appropriate time steps, that \(v_i^{1}\) approximates the solution of
\[
-v_t+v_x=1 \text{ on } (0,1)\times(-1,1), \qquad v = 0 \text{ on }
(0,T)\times\left\{-1\right\}\cup\left\{1\right\}\times(-1,1);
\]
while \(v_i^{-1}\) approximates the solution of
\[
-v_t-v_x=1 \text{ on } (0,1)\times(-1,1), \qquad v = 0 \text{ on }
(0,T)\times\left\{1\right\}\cup\left\{1\right\}\times(-1,1).
\]
Consequently, Assumption~\ref{ass:boundaryconv} is enforced by $v^{1,*}$ on $[0,1]\times\left\{-1\right\}$ and by $v^{-1,*}$ on $[0,1]\times\left\{1\right\}$.
\end{example}

Recall from Theorem~\ref{thm:discretewellposedness} that
\eb
0 \leq v_i \leq v_i^{\a} \quad\text{for all }\a \in A,
\ee
and note that by construction \(0 \le v_* \le v^*\). Assumption~\ref{ass:boundaryconv} thus implies that $v_*|_{[0,T] \times \pO} = v^*|_{[0,T] \times \pO} = 0$. Observe that because \eqref{eq:boundaryconv} holds in particular for all \((t,x) \in \{T\}\times\pO\), Assumption~\ref{ass:boundaryconv} implicitly enforces that the initial condition $v_T$ vanishes on $\pO$ as the $v_i^\alpha$ interpolate $v_T$ at the final time.

\begin{lemma}\label{lem:initalconverge}
The sub and super-solutions \(v^*\) and \(v_*\) satisfy 
\Eb\label{eq:initialconverge2} 
v^*(T,\cdot)=v_*(T,\cdot)=v_T \quad\text{ on } \, \oO.
\Ee
\end{lemma}

\begin{proof}
Fix $\eps > 0$ and choose a $v_T^\eps \in C^{\infty}(\R^{d})$ such that $v_T - 2 \eps \ge v_T^\eps \ge v_T - 3 \eps$. Owing to Assumption \ref{ass:ellproj} there is \(n\in\N\) such that $\| \P v_T^\eps - v_T^\eps \|_{L^{\infty}(\O)} \leq \eps$ and $\| \m I_i v_T - v_T \|_{L^{\infty}(\O)} \leq \eps$ for all $i \ge n$. Hence, for $i \ge n$, 
\begin{align} \label{eq:initalconverge3}
v_i(T, \cdot) = \m I_i v_T \geq \P v_T^\eps \geq v_T - 4 \eps.
\end{align}
Recalling \eqref{eq:ellprojdef} and as $v_T^\eps \in C^{\infty}(\R^{d})$, it is clear that there exists \(K = K(\eps) \geq 0\) which bounds
\begin{align*}
\bigl| \bigl( \bigl( \sia E+\sia I \bigr)\P v_T^\eps - \sia C \bigr)_{\ell} \bigr| = \bigl| - \bigl( \bia a(\il y) + \bbia a(\il y) \bigr) \bigl\langle \Delta v_T^\eps, \hil \phi \bigr\rangle + \bigl\langle \bigl( \bia b(\il y) + \bbia b(\il y) \bigr) \cdot \nabla \P v_T^\eps + \bigl( \bia c(\il y) + \bbia c(\il y) \bigr) \P v_T^\eps, \hil \phi \bigr\rangle - \bigl( \sia C \bigr)_{\ell} \bigr|
\end{align*}
for all \( i \in \N \), \(\ell \in \{1, \ldots, N\} \) and \(\a \in A \). Define $w_i = \P v_T^\eps - K (T-t)$. To show that \(v_i(\ik s,\cdot) \geq w_i(\ik s,\cdot)\) assume \(v_i(s_i^{k+1},\cdot) \geq w_i(s_i^{k+1},\cdot)\), noting \eqref{eq:initalconverge3} for $s_i^{k+1} = T$. Fix an $i$ and $\ell$ and let $\alpha = \alpha_i^{k,\ell}(v_i)$ as for \eqref{eq:max_alpha}. From 
\begin{align*}
 -d_i w_i (\ik s,\il y) + \bigl( \sia E w_i(s_i^{k+1},\cdot)+ \sia I w_i(s_i^{k+1},\cdot) \bigr)_{\ell}
& = -K + \left( \left(\sia E + \sia I\right)\P v_T^\eps \right)_{\ell} - K(T-s_i^{k+1})\pair{\bia c}{\hil \phi} - K(T-\ik s)\pair{\bbia c}{\hil\phi} \\
& \leq \left(\sia C\right)_{\ell} \stackrel{\eqref{numsol}}{=} - d_i v_i(\ik s,\il y) + \bigl( \si E^{v_i} v_i(s_i^{k+1},\cdot)+ \si I^{v_i} v_i(s_i^{k+1},\cdot) \bigr)_{\ell}
\end{align*}
we may deduce that
\[
 \left(\left(h_i \mathsf{I}_i^{v_i}+\Id\right) \left[ v_i(\ik s,\cdot)-w_i(\ik s,\cdot) \right]\right)_{\ell}
 \geq \left(\left(h_i\mathsf{E}_{i}^{v_i} - \Id\right) \left[ v_i(\ik s, \cdot)-w_i(\ik s ,\cdot)\right]\right)_{\ell} \geq 0.
\]
Note that \(v_i(s_i^k,\cdot) \in V_i^0\) vanishes on \(\pO\) and \(w_i(\ik s,\cdot) \leq 0\) on \(\pO\). Thus Lemma \ref{lem:monotonicity} and \eqref{eq:wDMP} imply \(v_i(\ik s,\cdot) \geq w_i(\ik s,\cdot)\) on \(\oO\). Because \(K\) is independent of \(i\) and \( \P v_T^\eps \tends v_T^\eps \) as $i \to \infty$, we have for any sequence \( \bigl(\ik s ,\il y\bigr) \tends \bigl(T,x \bigr)\), \(x \in \O\),
\begin{align*} 
\liminf_{i\tends \infty} v_i  \bigl(\ik s ,\il y\bigr) &\geq \liminf_{i\tends \infty} w_i\bigl(\ik s ,\il y\bigr) \geq v_T(x)- 4 \eps.
\end{align*}
So \(v_*(T,\cdot) \geq v_T - 4 \eps\). Since \(\eps\) was arbitrary, \( v_*(T,\cdot) \geq v_T\). The argument for showing that $v^* \le v_T$ is analogous with $w_i = P_i v_T^\eps + K (T- t)$ and $v_T + 2 \eps \le v_T^\eps \le v_T + 3 \eps$. To conclude, $v_T \le v_*(T,\cdot) \le v^*(T,\cdot) \le v_T$, which proves \eqref{eq:initialconverge2}.
\end{proof}

The proof of Lemma \ref{lem:initalconverge} is related to the arguments in \cite[p.\ 335]{FS06}. In the next assumption we draw upon one of the building blocks of the theory of viscosity solutions, namely the extension of classical comparison principles to spaces of semi-continuous functions, cf.~\cite[Sec.~5]{CIL92} and \cite[p.~219]{FS06}.

\begin{assumption} \label{ass:comp}
Let $\overline{v}$ be a lower semi-continuous supersolution with $\overline{v}|_{[0,T] \times \pO} = 0$ and $\overline{v}(T,\cdot) = v_T$.  Similarly, let $\underline{v}$ be an upper semi-continuous subsolution with $\underline{v}|_{[0,T] \times \O} = 0$ and $\underline{v}(T,\cdot) = v_T$. Then $\underline{v} \le \overline{v}$.
\end{assumption}

Let $t = \theta s_i^k + (1 - \theta) s_i^{k+1} \in [s_i^k, s_i^{k+1}]$ lie between two time steps, $\theta \in [0,1]$. Then we interpret $v_i(t, \cdot)$ as the linear interpolant between $v_i(s_i^k, \cdot)$ and $v_i(s_i^{k+1}, \cdot)$:
\begin{align} \label{eq:affine}
v_i(t, \cdot) = \theta v_i(s_i^k, \cdot) + (1 - \theta) v_i(s_i^{k+1}, \cdot).
\end{align}
\begin{theorem} \label{thm:uniform}
One has $v_* = v^* = v$, where \(v\) is the unique viscosity solution of equation \eqref{HJB} with \(v(T,\cdot)=v_T\) and \(\eval{v}{[0,T]\times\pO}=0\).  Furthermore
\begin{align} \label{conv}
\lim_{i \to \infty} \| v_i - v \|_{L^\infty((0,T) \times \O)} = 0.
\end{align}
\end{theorem}

\begin{proof}
The previous assumption implies that $v_* \ge v^*$ thus \(v^*=v_*=v\). Select for each $i \in \N$ a point $(t_i, x_i) \in [0,T] \times \oO$ such that
\[
\| v_i - v \|_{L^\infty((0,T)\times\O)} = |v_i - v|(t_i, x_i).
\]
Such $(t_i, x_i)$ exist as $v_i - v$ is a continuous function on a compact domain. Let $x_i$ belong to (the closure of) the element $T$ of the finite element mesh and $t \in [s_i^\kappa, s_i^{\kappa+1}]$; then $v_i(t_i, x_i)$ is a weighted average of the values of $v_i$ at the corners of the slab $[s_i^\kappa, s_i^{\kappa+1}] \times \overline{T}$. Thus there is a corner $(s_i^k, y_i^\ell)$ of the slab such that
\[
\| v_i - v \|_{L^\infty((0,T)\times\O)} \le |v_i(s_i^k, y_i^\ell) - v(t_i, x_i)|.
\]
If \eqref{conv} was wrong we could select a subsequence and an $\eps > 0$ such that
\[
\liminf_{j \to \infty} \, \bigl|v_{i(j)}(s_{i(j)}^k, y_{i(j)}^\ell) -
v(t_{i(j)}, x_{i(j)}) \bigr| \ge \eps.
\]
By possibly passing to a further subsequence we may assume that $\{ (t_{i(j)}, x_{i(j)}) \}_j$ converges to an $(t,x) \in [0,T] \times \oO$.  However, this contradicts
\[
v(t,x) = v_*(t,x) \le \liminf_{j \to \infty} v_{i(j)}(s_{i(j)}^k, y_{i(j)}^\ell)
\le \limsup_{i \to \infty} v_{i(j)}(s_{i(j)}^k, y_{i(j)}^\ell) \le v^*(t,x) =
v(t,x).
\]
Thus \eqref{conv} holds. \end{proof}

\section{Gradient convergence}

For shorthand, let $W=W^{1,\infty}((0,T)\times\O)$. It is convenient to introduce the discrete spaces
\[
W_i := \{ v \in C([0,T], V_i^0) : v|_{[\ik s, \iko s] \times \O} \text{ is affine in time} \},
\]
which means that functions in $W_i$ have between two time-steps the form of \eqref{eq:affine}. Observe that $W_i \subset W$ for all $i \in \N$.

Fix an arbitrary $\alpha \in A$. It is convenient to view $\sia E$ and $\sia I$ as bilinear forms on $H^1(\O) \times V_i$. Functions $u \in V_i$ have the nodal representation
\[
u(y) = \sum_{\ell} u(\il y) \, \il \phi(y).
\]
To test with functions other than $\hil \phi$ we introduce the following bilinear form as a partially discrete pivot: for $w \in H^1(\O)$ and $u \in V_i$
\[
\llangle \sia E w, u \rrangle := \sum_\ell u(\il y) \bigl( \bia  a(\il y) \langle \nabla w, \nabla \il \phi \rangle + \langle \bia  b \cdot \nabla w + \bia  c \, w, \il \phi \rangle \bigr).
\]
We use corresponding interpretation for $\llangle \sia I w, u \rrangle$ and also
\[
\llangle w, u \rrangle = \llangle \Id w, u \rrangle = \sum_\ell w(\il y) \, u(\il y)\| \il \phi \|_{L^1(\O)} \qquad \text{and} \qquad \llangle \sia C, u \rrangle = \sum_\ell u(\il y) \, \langle d_i^\alpha, \il \phi \rangle = \langle d_i^\alpha, u \rangle.
\]
Assume that for the chosen $\alpha$:
\begin{align} \nonumber | w |_{L^2([0,T], H^1(\O))}^2
\lesssim \, & \sum_{k = 0}^{(T / h_i) - 1} \Bigl( \bllangle \bigl(h_i \sia E - \Id \bigr) w(\iko s,\cdot) + \bigl(h_i \sia I + \Id \bigr) w(\ik s,\cdot), w(\ik s,\cdot) \brrangle \Bigr) + \tst \oh \llangle w(T,\cdot), w(T,\cdot) \rrangle + \| w(T, \cdot) \|_{H^1(\O)}^2\\ \label{eq:posdef}
\stackrel{(*)}{=} \, & \sum_{k = 0}^{(T / h_i) - 1} \Bigl( h_i \bllangle \sia E w(\iko s,\cdot) + \sia I w(\ik s,\cdot), w(\ik s,\cdot) \brrangle \tst + \oh \llangle w(\iko s,\cdot) - w(\ik s,\cdot), w(\iko s,\cdot) - w(\ik s,\cdot) \rrangle \Bigr)\\
& + \tst \oh \llangle w(0,\cdot), w(0,\cdot) \rrangle + \| w(T, \cdot) \|_{H^1(\O)}^2 \nonumber
\end{align}
for all $w \in W_i$ with $w \geq 0$ and $i \in \N$, where $(*)$ is a simple reformulation in terms of a telescope sum.

Due to the definition of the numerical method and the non-negativity of the $v_i$,
\begin{align*}
| v_i |_{L^2([0,T], H^1(\O))}^2 \lesssim \, & \sum_{k = 0}^{(T / h_i) - 1} \Bigl( \bllangle \bigl(h_i \sia E - \Id \bigr) v_i(\iko s,\cdot) + \bigl(h_i \sia I + \Id \bigr) v_i(\ik s,\cdot), v_i(\ik s,\cdot) \brrangle \Bigr) + \tst \oh \llangle v_i(T,\cdot), v_i(T,\cdot) \rrangle + \| v_i(T, \cdot) \|_{H^1(\O)}^2\\
\le \, & \sum_{k = 0}^{(T / h_i) - 1} \bllangle h_i \sia C, v_i(\ik s,\cdot) \brrangle + \tst \oh \llangle v_i(T,\cdot), v_i(T,\cdot) \rrangle  + \| v_i(T, \cdot) \|_{H^1(\O)}^2\\
\lesssim \; & T \, \| d_i^\alpha \|_{L^1(\O)} \, \| v_i \|_{L^\infty([0,T] \times \O)} + \| v_i(T,\cdot) \|_{H^1(\O)}^2.
\end{align*}
Thus, with the $L^\infty$ control established in the previous section, it is apparent that the $v_i$ are bounded in $L^2([0,T], H^1(\O))$ provided that $v_i(T,\cdot)=\calI_iv_T$ are bounded in $H^1(\O)$; this condition holds if $v(T,\cdot) \in W^{1,\infty}(\O)$. The first convergence result for the gradient is therefore that, owing to the Banach-Alaoglu theorem, $v_i \wc v$ weakly in $L^2([0,T], H^1(\O))$, using $L^\infty((0,T) \times \O)$ convergence to pass from $L^2([0,T], H^1(\O))$ weak convergence of subsequences to $L^2([0,T], H^1(\O))$ weak convergence of the whole sequence.

The question arises under which circumstances the convergence in the gradient is also strong. We demonstrate this under the below Assumption \ref{ass:grad_reg}. We note that supposing \eqref{eq:posdef} points towards uniform ellipticity of $L^\alpha$. Let $\Lambda_0$ be the level set $\{ (t,x) \in (0,T) \times \O : v(t,x) = 0 \}$. For a smooth $v$ the boundary of $\Lambda_0$ is always a $d-1$ dimensional set if $0$ is a regular value.

\begin{assumption} \label{ass:grad_reg}
The value function $v$ belongs to the space $W =  W^{1,\infty}( (0,T) \times \O)$ and the $d$-dimensional Lebesgue measure of the boundary of \,$\Lambda_0$ vanishes: ${\rm vol} (\partial \Lambda_0) = 0$. The coefficients $\bia a$ and $\bbia a$ belong to $W^{1,\infty}(\O)$ and \eqref{eq:posdef} is satisfied.
\end{assumption}

Let us suppose momentarily that there are approximations $Q_i v \in W_i$ to $v$ such that
$Q_i v \le v_i$ for all $i \in \N$ and
\[
\lim_{i \to \infty} \| v - Q_i v \|_{L^2([0,T], H^1(\O))} = 0,
\]
and
\begin{align} \label{eq:projcond}
\lim_{i \to \infty} \sum_{k = 0}^{(T / h_i) - 1} \bllangle \bigl(h_i \sia E - \Id \bigr) Q_i v(\iko s,\cdot) + \bigl(h_i \sia I + \Id \bigr) Q_i v(\ik s,\cdot), (v_i - Q_i v)(\ik s, \cdot) \brrangle \to 0.
\end{align}
We will construct such $Q_i v$ below. With $\xi^k = v_i(\ik s,\cdot) - Q_i v(\ik s,\cdot)$,
\begin{align} \nonumber
& | v_i - Q_i v |_{L^2([0,T], H^1(\O))}^2 \lesssim \sum_{k = 0}^{(T / h_i) - 1} \bllangle \bigl(h_i \sia E - \Id \bigr) \xi^{k+1} + \bigl(h_i \sia I + \Id \bigr) \xi^k, \xi^k \brrangle\\ \nonumber
= \, & \sum_{k = 0}^{(T / h_i) - 1} \bllangle \bigl(h_i \sia E - \Id \bigr) v_i(\iko s,\cdot) + \bigl(h_i \sia I + \Id \bigr) v_i(\ik s,\cdot), \xi^k \brrangle -
\sum_{k = 0}^{(T / h_i) - 1} \bllangle \bigl(h_i \sia E - \Id \bigr) Q_i v(\iko s,\cdot) + \bigl(h_i \sia I + \Id \bigr) Q_i v(\ik s,\cdot), \xi^k \brrangle\\
\stackrel{(*)}{\le} \, & \sum_{k = 0}^{(T / h_i) - 1} \bllangle h_i \sia C, \xi^k \brrangle \; -  \label{eq:Hone_split}
\sum_{k = 0}^{(T / h_i) - 1} \bllangle \bigl(h_i \sia E - \Id \bigr) Q_i v(\iko s,\cdot) + \bigl(h_i \sia I + \Id \bigr) Q_i v(\ik s,\cdot), \xi^k \brrangle,
\end{align}
using in $(*)$ the numerical scheme, $\xi^{T/h_i} = 0$ and that, due to the assumptions on the $Q_i$, the sign of $v_i - Q_i v$ is known. Since
\begin{align*}
\sum_{k = 0}^{(T / h_i) - 1} \bllangle h_i \sia C, \xi^k \brrangle \le \, & \| d_i^\alpha \|_{L^2(\O)} \, \sum_{k = 0}^{(T / h_i) - 1} h_i \bigl (\| v_i(\ik s, \cdot) - v(\ik s, \cdot) \|_{L^2(\O)} + \| v(\ik s, \cdot) - Q_i v(\ik s, \cdot) \|_{L^2(\O)} \bigr)\\
\lesssim \, & \| d_i^\alpha \|_{L^2(\O)} \, \bigl( \| v_i - v \|_{L^2((0,T) \times \O)} + \| v - Q_i v \|_{L^2((0,T) \times \O)} \bigr),
\end{align*}
the first term in \eqref{eq:Hone_split} vanishes as $i \to \infty$. The second term vanishes due to \eqref{eq:projcond}. Hence $| v_i - v |_{L^2([0,T], H^1(\O))} \to 0$ as $i \to \infty$.

\begin{theorem}
If there is an $\alpha \in A$ such that Assumption \ref{ass:grad_reg} holds, then the numerical solutions converge to the exact solution strongly in $L^2([0,T], H^1(\O))$.
\end{theorem}

\begin{proof}
It remains to show that suitable $Q_i$ can be constructed, given Assumption \ref{ass:grad_reg}. Denoting the nodal interpolant on $[0,T] \times \oO$ by $\m I_i$ we define
\begin{align} \label{eq:defQ}
Q_i : \; W \to W_i, \; w \mapsto \m I_i \max \{ w - \| v - v_i \|_{L^\infty((0,T) \times \O)} , 0 \}.
\end{align}
Observe that the $\max$ operator in \eqref{eq:defQ} switches between the first and second argument in the vicinity of $\partial \Lambda_0$ for $i$ sufficiently large. Furthermore, $Q_i v \in W_i$ satisfies homogeneous boundary conditions and $Q_i v \le v_i$ and, by the mean value theorem,
\[
\| Q_i v \|_{W^{1,\infty}((0,T) \times \O)} \le \| v \|_{W^{1,\infty}((0,T) \times \O)}.
\]
Note also that for all nodes $\il y$ and time levels $\ik s$
\[
 0 \leq \left(v_i-Q_i v\right)(\ik s,\il y) = \min\bigl\{ \left(v_i-v\right)(\ik s,\il y)+\norm{v_i-v}_{L^{\infty}((0,T)\times\O)}, v_i(\ik s,\il y)\bigr\} \leq 2 \norm{v_i-v}_{L^{\infty}((0,T)\times\O)}.
\]
Consider the set $\Gamma_i$ of points which is not `affected by the cut-off below $0$' in \eqref{eq:defQ} in the sense that
\[
\Gamma_i := \bigl\{ (t,x) \in (0,T) \times \O : \inf_{j \ge i} Q_j v(t,x) > 0 \; \text{ or } (t,x) \; \in \Lambda_0 \bigr\}.
\]
The set $\Gamma_i'$ contains the points which are at least one element's length away from the boundary of $\Gamma_i \setminus \partial \Lambda_0$:
\[
\Gamma_i' := \bigl\{ (t,x) \in \Gamma_i : \{ (s,y) \in (0,T) \times \O : \| (t,x) - (s,y) \| < \sup_{j \ge i} h_j + (\Delta x)_j \}\subset \Gamma_i \setminus \partial \Lambda_0 \bigr\}.
\]
Notice that $\Gamma_i$ and $\Gamma_i'$ are hierarchical families. Since $\| v - v_i \|_{L^\infty((0,T) \times \O)} \to 0$ and $h_i + (\Delta x)_i \to 0$ as $i \to \infty$ it follows that
\[
\bigcup_{i \in \N} \Gamma_i' = \bigl( (0,T) \times \O \bigr) \setminus \partial \Lambda_0.
\]
Crucially, $(\partial_t Q_j v)|_{\Gamma_i'} = (\partial_t \m I_j v)|_{\Gamma_i'}$ and $(\nabla Q_j v)|_{\Gamma_i'} = (\nabla \m I_j v)|_{\Gamma_i'}$ for $j \ge i$.

For every $\eps > 0$ there are $i,j \in \N$ such that ${\rm vol}(\O \setminus \Gamma_i') \le \eps^2$ and $\| Q_k v - v \|_{H^1(\Gamma_i')} \le \eps$ for all $k \ge j$. Therefore
\[
\| Q_k v - v \|_{H^1((0,T) \times \O)} \lesssim \| Q_k v - v \|_{H^1(\Gamma_i')} + \sqrt{{\rm vol}(\O \setminus \Gamma_i')} \, \| v \|_{W^{1,\infty}((0,T) \times \O)} \le \eps (1 + \| v \|_{W^{1,\infty}((0,T) \times \O)}),
\]
giving strong convergence in $H^1((0,T) \times \O)$, meaning convergence in the spatial gradient and the time derivative. The terms connected to the time derivative in \eqref{eq:projcond} vanish in the limit as
\[
\sum_{k = 0}^{(T / h_i) - 1} \bllangle Q_i v(\iko s,\cdot) - Q_i v(\ik s,\cdot), \xi^k \brrangle = \sum_{k = 0}^{(T / h_i) - 1} h_i \bllangle (\partial_t Q_i v)|_{(\ik s,\iko s)}, \xi^k \brrangle \lesssim \| \partial_t v \|_{L^2((0,T) \times \Omega)} \, \| \xi^k \|_{L^2((0,T) \times \Omega)}.
\]
Recall that 
\begin{align*}
\llangle \sia I Q_i v(\ik s, \cdot), \xi^k \rrangle & = \sum_\ell (v_i - Q_i v)(\ik s, \il y) \bigl( \bbia  a(\il y) \langle \nabla Q_i v(\ik s, \cdot), \nabla \il \phi \rangle + \langle \bbia  b \cdot \nabla Q_i v(\ik s, \cdot) + \bbia  c \, Q_i v(\ik s, \cdot), \il \phi \rangle \bigr).
\end{align*}
The lower-order terms vanish due to the uniform convergence of $ v_i -Q_i v$ to $0$ and the bound
\[
\sup_i \| \bbia  b \cdot \nabla Q_i v(\ik s, \cdot) + \bbia  c \, Q_i v(\ik s, \cdot) \|_{L^\infty(\O)} < \infty.
\]
We note for the second-order term that
\begin{align*}
& \sum_\ell (v_i - Q_i v)(\ik s, \il y) \bbia  a(\il y) \langle \nabla Q_i v(\ik s, \cdot), \nabla \il \phi \rangle = \, \langle \nabla Q_i v(\ik s, \cdot), \nabla \m I_i (\bbia a(v_i - Q_i v))(\ik s, \cdot) \rangle,
\end{align*}
so that in \eqref{eq:projcond} the implicit part of the second-order term becomes
\begin{align} \label{eq:implicit}
& \sum_{k = 0}^{(T / h_i) - 1} \!\!\! h_i \, \langle \nabla Q_i v(\ik s, \cdot), \nabla \m I_i (\bbia a(v_i - Q_i v))(\ik s, \cdot) \rangle
= \, \int_0^T \! \langle \m J_i \nabla Q_i v, \m J_i\nabla \m I_i (\bbia a(v_i - Q_i v))\rangle \, \d t,
\end{align}
where $\m J_i$ maps any $w : [0,T] \to L^2(\O;\R^d)$ onto the step function with $(\m J_i w)|_{[\ik s, \iko s)} \equiv w(\ik s, \cdot)$. Note that $\m J_i \nabla Q_i v$ converges strongly in $L^2((0,T) \times \O; \R^d)$. At a time $\ik s \in [0,T)$ the bound
\begin{align*}
\| \nabla \m I_i (\bbia a(v_i - Q_i v)) \|_{L^2(\O; \R^d)} \lesssim \, & \| \nabla \m I_i (\bbia a v_i) \|_{L^2(\O; \R^d)} + \| \bbia a Q_i v \|_{W^{1,\infty}(\O)} \lesssim \| \bbia a \|_{W^{1,\infty}(\O)} \cdot \left( \| v_i \|_{H^1(\O)} + \|  v \|_{W^{1,\infty}(\O)}\right)
\end{align*}
follows from an inverse estimate and
\[
\sum_T \| \nabla \m I_i (\bbia a v_i) \|_{L^2(T; \R^d)}^2 \lesssim \sum_T (\Delta x)_T^d \, \| \nabla \m I_i (\bbia a v_i) \|_{L^\infty(T; \R^d)}^2
\lesssim \sum_T \| \bbia a \|_{W^{1,\infty}(T)}^2 \; \bigl( (\Delta x)_T^d \, \| v_i \|_{W^{1,\infty}(T)}^2 \bigr).
\]
The convergence 
\[
\lim_{i \to \infty} \; \int_0^T \langle w, \m J_i \nabla \m I_i (\bbia a(v_i - Q_i v))\rangle \d t = - \lim_{i \to \infty} \; \int_0^T \langle \nabla \cdot w, \m J_I \m I_i (\bbia a(v_i - Q_i v))\rangle \d t = 0
\]
with test functions $w$ in the dense subset $C^1((0,T) \times \O; \R^d)$ gives weak convergence of $\nabla \m I_i (\bbia a(v_i - Q_i v))$ in $L^2((0,T) \times \O; \R^d)$, see \cite[p.~121]{Yoshida}. Combing weak and strong convergence \cite[Prop.~21.23]{ZeidlerII}, it is ensured that \eqref{eq:implicit} converges to $0$ as $i \to \infty$. A similar argument guarantees that $\sum_k h_i \, \llangle \sia E Q_i v(\iko s, \cdot), \xi^k \rrangle$ vanishes in the limit.
\end{proof}

The regularity of the exact value function $v$ is, for instance, discussed in Section IV.8 and IV.9 of \cite{FS06}. Another item of Assumption \ref{ass:grad_reg}, namely the justification of \eqref{eq:posdef}, is examined in the following example:

\begin{example}
a) Suppose that $a^\alpha$ is positive and constant and, for all smooth $w$,
\[
L^\alpha w = I^\alpha w = - a^\alpha \Delta w + b^\alpha \cdot \nabla w + c^\alpha w, \qquad E^\alpha w = 0, 
\]
and, to obtain semi-definiteness in the lower-order terms, $\tst c^{\a} - \oh \nabla \cdot b^{\a} \ge 0$. Then, for $w \in W_i$,
\begin{align*} \nonumber
a^\alpha | w |_{L^2([0,T], H^1(\O))}^2
\lesssim \; & a^\alpha \! \sum_{k = 0}^{(T / h_i) - 1} \! h_i \bigl\langle \nabla w(\ik s,\cdot), \nabla w(\ik s,\cdot) \bigr\rangle + \| w(T,\cdot) \|_{H^1(\O)}^2 = \sum_{k = 0}^{(T / h_i) - 1} \!\! h_i \bllangle \sia I w(\ik s,\cdot), w(\ik s,\cdot) \brrangle + \| w(T,\cdot) \|_{H^1(\O)}^2.
\end{align*}
b) Suppose that $a^{\a}\in W^{2,\infty}(\O)$ is non-constant, positive, uniformly bounded from below and that $c^{\a}-\half (\nabla \cdot b^{\a}+\Delta a^{\a}) \geq 0$, noting for smooth $w$:
\[
\langle L^\alpha w, w \rangle = \tst \langle a^\alpha \nabla w, \nabla w \rangle + \langle (c^{\a}-\half (\nabla \cdot b^{\a}+\Delta a^{\a})) w, w \rangle.
\]
Again choosing a fully implicit scheme with $L^\alpha = I^\alpha$, the highest order term in $\llangle \sia I w, w \rrangle$ is at time $\ik s$:
 \[
\sum_{\ell} w(\ik s, \il y) a^{\a}(\ik s, \il y) \pair{\nabla w(\ik s, \cdot)}{\nabla \il \phi}=\pair{\nabla w(\ik s, \cdot)}{\nabla \calI_i (a^{\a}(\ik s, \cdot) w(\ik s, \cdot))}.
\]
According to Theorem 2.1 in \cite{Demlow11} there is a constant $C = C\bigl(\| a^\a \|_{W^{2,\infty}(\O)}\bigr)$ such that for $i$ sufficiently large
 \[
\pair{\nabla w}{\nabla \calI_i (a^{\a} w)} - \pair{\nabla w}{\nabla a^{\a} w} \le \| \nabla w \|_{L^2(\O; \R^d)} \cdot \| \calI_i (a^{\a} w) - a^{\a} w \|_{H^1(\O)} \le C \; (\Delta x)_i \, \| w \|_{H^1(\O)}^2,
 \]
using that the $\eta$ appearing in the proof in \cite{Demlow11} is defined in terms of nodal interpolation. It then follows from Poincar\'{e}'s inequality that there is some $C$ such that for $C (\Delta x)_i < \oh \inf_{\O} a^{\a}$ that $\abs{w}^2_{H^1(\O)} \lesssim \llangle \sia I w, w \rrangle$ for $w \in V_i^0$, implying \eqref{eq:posdef}.
\end{example}

\section{Example: the method of artificial diffusion}

The purpose of this section is to provide a way of constructing the operators \(\sia E\) and \(\sia I\) in order to satisfy Assumptions~\ref{consistency} and \ref{monotonicity}. This approach, called the method of artificial diffusion, is based on the fact that for strictly acute meshes, the discrete Laplacian is monotone. Further details on the method of artificial diffusion and monotone finite element schemes may, for example, be found in \cite{BE02}, \cite{Codina}  and \cite{XZ99}.

Let \(\m T_i\) be the mesh corresponding to the finite element space \(V_i\). Given a function $f : \Omega \to \R^d$ we denote
\[
|f|_T:= \Bigl( \sum_{j=1}^{d} \bigl\| f_j \bigr\|_{L^\infty(T)}^2 \Bigr)^{\half}, \qquad T \in \m T_i,\;i\in\N.
\]
If $f$ is elementwise constant then $|f|_T$ is simply the Euclidean norm of $f$ on $T$. Let \((\Delta x)_T\) denote the diameter of $T$. We assume that the meshes \(\calT_i\) are strictly acute \cite{BE02} in the sense that there exists \(\theta \in (0,\pi/2)\) such that
\Eb\label{eq:acutemesh}
\nabla \il \phi \cdot \nabla \phi_{i}^{l} \bigl|_T \leq - \, \sin(\theta) \; | \nabla \il \phi |_T \; | \nabla \phi^{l}_i |_T \qquad \forall \ell,l \leq N \; \forall i \in \N.
\Ee
We choose a splitting of the form $a^{\a} = \tia a + \ttia a$, $b^{\a} = \bia b + \bbia b$, $c^{\a} = \bia c + \bbia c$ and $d^{\a}=d_i^{\a}$, where all terms are in \(C(\oO)\), \(\tia a\) and \(\ttia a\) are non-negative and all \( \bia c\) and \( \bbia c\) are non-negative and satisfy inequality \eqref{react}. Choose non-negative \( \bial \nu \) and \( \bbial \nu \) such that for all $T$ which have $\il y$ as vertex:
\begin{subequations}\label{ineq:nut}
\begin{align}
\bigl( |  \bia b |_T \, +  (\Delta x)_T \|  \bia c \|_{L^\infty(T)} \bigr) \le \, &  \bial \nu \, \sin(\theta) \, | \nabla \hil \phi |_T \, {\rm vol}(T), \\
\bigl( | \bbia b |_T \, +  (\Delta x)_T \| \bbia c \|_{L^\infty(T)} \bigr) \le \, & \bbial \nu \, \sin(\theta) \, | \nabla \hil \phi |_T \, {\rm vol}(T).
\end{align}
\end{subequations}
Choose \( \bia a \) and \(\bbia a \) both in \(C(\oO)\) such that $\bia a(\il y) \geq \max \bigl\{ \tia a(\il y), \bial \nu \bigr\}$ and $\bbia a(\il y) \geq \max \bigl\{ \ttia a(\il y), \bbial \nu \bigr\}$. Now suppose that \( w \in V_i\) has a non-positive minimum at an interior node \( \il y\). By extending the arguments of \cite{BE02}, we show that
\begin{align} \label{av_mono}
(\sia E w)_\ell \le 0, \qquad (\sia I w)_\ell \le 0.
\end{align}
We illustrate the proof of \eqref{av_mono} for the implicit term. From the strict acuteness condition on the mesh, it can be shown that on the restriction to $T$ \cite[Lemma 3.1]{BE02}
\eb
\nabla w \cdot \nabla \il \phi = \cos \bigl( \angle (\nabla w, \nabla \il \phi) \bigr) \, | \nabla w |_T \, | \nabla \il \phi |_T \leq - \sin(\theta) | \nabla w |_T \, | \nabla \il \phi |_T.
\ee
Using $\bbia c \ge 0$, $w(\il y) \le 0$ and $\| \hil \phi \|_{L^1(\O)}=1$,
\begin{align*}
\langle \bbia c \, w, \hil \phi \rangle = \, & \int_\O \bbia c(x) \, \bigl( w(\il y) + \nabla w(x) \cdot (x - \il y) \bigr) \, \hil \phi(x) \, \d x\\
\le \, & \int_\O \bbia c(x) \, \nabla w(x) \cdot (x - \il y) \, \hil \phi(x) \, \d x \le \sum_T  \, \| \bbia c \|_{L^\infty(T)} \, |\nabla w|_T \, (\Delta x)_T.
\end{align*}
Consequently,
\begin{align*}
(\sia I w)_\ell = \, & \bbia a(\il y) \langle \nabla w, \nabla \hil \phi \rangle + \langle \bbia b \cdot \nabla w + \bbia c \, w, \hil \phi \rangle\\
\le \, & \sum_T -  \bbia a(\il y) \sin(\theta) | \nabla w |_T \, | \nabla \hil \phi |_T \, {\rm vol}(T) + |\bbia b |_T \, | \nabla w |_T + \| \bbia c \|_{L^\infty(T)} \, |\nabla w|_T \, (\Delta x)_T \\
\le \, &  \sum_T | \nabla w |_T \bigl( \bigl( | \bbia b |_T \, +(\Delta x)_T \| \bbia c \|_{L^\infty(T)} \bigr) - \bbial \nu \, \sin(\theta) \, | \nabla \hil \phi |_T \, {\rm vol}(T) \bigr) \le 0.
\end{align*}
The proof of $(\sia E w)_\ell \le 0$ is analogous. As hat functions $\il \phi$ attain a non-positive minimum at all $y_i^j$ where $j \neq \ell$, all off-diagonal entries of $\sia E$ are non-positive. Hence with a suitable time step restriction the $h_i \sia E - \Id$ are monotone, which ensures that Assumption \ref{monotonicity} is satisfied. 

The scaling of the terms in \eqref{ineq:nut} with respect to $(\Delta x)_T$ leads to Assumption \ref{consistency}. Due to shape-regularity all elements $T$ on a patch are of comparable size; giving $\| \il \phi \|_{L^1(\O)} \leq C \, {\rm vol}(T)$ for all $T\subset \supp \il \phi$ with a constant $C$ which is independent of $h$ and $\ell$. Hence in \eqref{ineq:nut}, we see that
\[
{\rm vol}(T) \; | \nabla \hil \phi |_T \geq \frac{{\rm vol}(T)}{(\Delta x)_T \, \| \il \phi \|_{L^1(\O)}} \geq \frac{1}{\vphantom{(\Delta x)_T \, \| \il \phi \|_{L^1(\O)}} C (\Delta x)_T}.
\]
Thus, if $\bial \nu$ and $\bbial \nu$ are chosen optimally then for $T \subset \supp \il \phi$
\begin{align} \label{eq:artdifforder}
\bial \nu =  \mathsf{O} \big( \sup_{T} \bigl\{ | \bia b|_T (\Delta x)_T + \| \bia c \|_{L^\infty(T)} (\Delta x)_T^2 \bigr\} \big), \qquad
\bbial \nu =  \mathsf{O}\big( \sup_{T} \bigl\{ | \bbia b|_T (\Delta x)_T + \| \bbia c \|_{L^\infty(T)} (\Delta x)_T^2 \bigr\} \big).
\end{align}
With \eqref{eq:artdifforder} in mind we return to the time step restriction for semi-implicit and explicit methods. The non-positivity of the diagonal terms of $h_i \sia E - \Id$ expands to
\begin{align*}
1 \ge \, & h_i \Bigl( \bia  a(\il y) \langle \nabla \il \phi, \nabla \hil \phi \rangle + \langle \bia  b \cdot \nabla \il \phi + \bia  c \, \il \phi, \hil \phi \rangle \Bigr)
= h_i \Bigl( {\mathsf O} \bigl( \bia a \, (\Delta x)_T^{-2} \bigr) + {\mathsf O} \bigl( | \bia b |_T \, (\Delta x)_T^{-1} \bigr) + {\mathsf O} \bigl( \bia c \bigr) \Bigr).
\end{align*}
Therefore the time step restriction imposed by $L^\alpha$ is $h_i \lesssim \sup_T \bigl( (\Delta x)_T^2 / \bia a(\il y)\bigr)$, $\il y \in \overline T$, if there is a non-zero $\tia a$ and $i$ is large. It is $h_i \lesssim \sup_T \bigl( (\Delta x)_T / |\bia b(\il y)|_T \bigr)$ if all $\bia a = 0$, $i \in \N$, and there are non-zero $\bia b$, and is $\mathsf{O}(1)$ if all $\bia a$ and $\bia b$ vanish. There is no restriction if also all $\bia c$ are zero.


\end{document}